\documentclass[11pt, a4paper]{article}
\usepackage{amsfonts,amssymb,amsmath,amscd,latexsym,makeidx,theorem}
\usepackage{hyperref}

\newtheorem{thm}{Theorem}[section]
\newtheorem{lem}[thm]{Lemma}
\newtheorem{prop}[thm]{Proposition}

\newtheorem{defn}{Definition}[section]

\newenvironment{proof}{\noindent\emph{Proof.}}{\hfill$\square$\medskip}

\newcommand{\rn}{\mathbb{R}^{n}}

\newcommand{\s}{\mathcal{S}}
\newcommand{\Pg}{P^n_{g_0}}

\DeclareMathOperator{\sign}{sign}

\title{Existence of entire solutions to a fractional Liouville equation in $\rn$}

\author{Ali Hyder \thanks{The author is supported by the Swiss National Science Foundation.} 
\\ {\small Universit\"at Basel}\\ {\small \texttt{ali.hyder@unibas.ch}} }

\begin{document}

\maketitle

\begin{abstract}
We study the existence of solutions to the problem
$$ (-\Delta)^{\frac{n}{2}}u=Qe^{nu}\quad\text{in }\rn,\quad V:=\int_{\rn}e^{nu}dx<\infty,$$
where $Q=(n-1)!$ or $Q=-(n-1)!$. Extending the works of Wei-Ye and Hyder-Martinazzi to arbitrary odd dimension $n\geq 3$ we show that to a certain extent the asymptotic behavior of $u$
 and the constant $V$ can be prescribed simultaneously. Furthermore if $Q=-(n-1)!$ then $V$ can be chosen to be any positive number. This is in contrast to the case $n=3$, $Q=2$, where Jin-Maalaoui-Martinazzi-Xiong showed that necessarily $V\le |S^3|$, and to the case $n=4$, $Q=6$, where C-S. Lin showed that $V\le |S^4|$. % It also extends previous results by Martinazzi and Hyder-Martinazzi. 
\end{abstract}

 \section{Introduction to the problem}
 In this paper we consider the equation 
\begin{align} \label{main-eq}
 (-\Delta)^{\frac{n}{2}}u=(n-1)!e^{nu}\quad \text{in $\rn$},
\end{align}
where $n\ge 1$ and 
\begin{align}\label{V}
 V:=\int_{\rn}e^{nu}dx<\infty.
\end{align}
The operator $(-\Delta)^{\frac{n}{2}}$ can be defined as follows. %is understood as $(-\Delta)^{\frac{n}{2}}:=(-\Delta)^{\frac{1}{2}}\circ(-\Delta)^{\frac{n-1}{2}}$ 
%for $n\geq 3$ odd integer where
For $s>0$ we set
$$\widehat{(-\Delta)^s \varphi }(\xi):=|\xi|^{2s}\hat{\varphi}(\xi),\quad \text{for }\varphi\in \s(\rn),$$
where
$$\mathcal{S}(\rn):=\left\{u\in C^{\infty}(\rn): \sup_{x\in\rn}|x|^N|D^{\alpha}u(x)|<\infty \text{ for all } N\in\mathbb{N} \text{ and } \alpha\in
 \mathbb{N}^n\right\}$$ is the Schwartz space. 
Define the space 
% In order to define the left hand side of \eqref{main-eq} first we recall the space 
$$L_{s}(\rn):=\left\{v\in L^1_{loc}(\rn):\int_{\rn}\frac{|v(x)|}{1+|x|^{n+2s}}dx<\infty\right\}.$$ 
Then we have the following definition:

\begin{defn}\label{sol}
Let $f\in L^1(\rn)$.
A function $u\in L_{\frac n2}(\rn)$ is said to be a solution of 
$$(-\Delta)^{\frac{n}{2}}u=f \quad \text{in }\rn,$$if 
 \begin{align}\label{def}
  \int_{\rn}u(-\Delta)^{\frac{n}{2}}\varphi dx=\int_{\rn}f\varphi dx, \text{ for every }\varphi\in\s(\rn),
 \end{align}
where the integral on the left-hand side of \eqref{def} is well-defined thanks to Proposition \ref{ests} in the appendix.
\end{defn}
%Since any solution of \eqref{main-eq}-\eqref{V} in the sense of Definition \ref{sol} is smooth (see...), now we assume $u$ to be smooth. 
 
 Geometrically if $u$ is a smooth solution of \eqref{main-eq}-\eqref{V} then the conformal metric $g_u:=e^{2u}|dx|^2$ on $\rn$ ($|dx|^2$ is the 
Euclidean metric on $\rn$) has constant $Q$-curvature equal to $(n-1)!$. Moreover the volume and the total $Q$-curvature of the metric $g_u$ are
$V$ and $(n-1)!V$ respectively. 

It is well-known that   
\begin{equation}\label{uspher}
u_{\lambda,x_0}(x):=\log\frac{2\lambda}{1+\lambda^2|x-x_0|^2},
\end{equation}
is a solution of \eqref{main-eq}-\eqref{V} with $V=|S^n|$ for every $\lambda>0$ and $x_0\in\rn$. 
Any such $u_{\lambda,x_0}$ is called spherical solution, because it can be obtained by pulling back the round metric of $S^n$ onto $\rn$ via stereographic projection. When $n=2$ W. Chen-C. Li \cite{Chen-Li} showed that these are the only solutions to \eqref{main-eq}-\eqref{V}. On the other hand in higher even dimension non-spherical
solutions do exist as shown by A. Chang-W. Chen \cite{CC}:

\medskip
\noindent \textbf{Theorem A (\cite{CC})}
\emph{Let $n\geq 4$ be any even number. Then for every $V\in (0,|S^n|)$ there exists at least one solution to \eqref{main-eq}-\eqref{V}. %Moreover in dimension four 
%for every $V\in (0,|S^n|)$ radial solutions do exist  and assume the asymptotic behavior $-a|x|^2$ at infinity.
}

\medskip

Moreover J. Wei and D. Ye \cite{W-Y} in dimension $4$ and A. Hyder and L. Martinazzi \cite{H-M} for arbitrary even dimension $n\geq 4$ proved the existence of solution to \eqref{main-eq}-\eqref{V}
with prescribed volume and asymptotic behavior in the following sense:

\medskip
\noindent \textbf{Theorem B (\cite{W-Y}, \cite{H-M})}
\emph{Let $n\ge 4$ be even. For a given  $V\in (0,|S^n|)$, and a given polynomial $P$ such that degree$(P)\leq n-2$ and
\begin{equation}\label{condP}
x\cdotp \nabla P(x)\to\infty \quad \text{as } |x|\to\infty,
\end{equation}
there exists a solution $u$ to \eqref{main-eq}-\eqref{V} having the asymptotic behavior
\begin{align}\label{asmp}
 u(x)=-P(x)-\alpha\log|x|+C+o(1),
\end{align}
where $\alpha:=\frac{2V}{|S^n|}$, and  $o(1)\to 0$ as $|x|\to\infty$.}

\medskip

When $n$ is odd things are more complex as the operator $(-\Delta)^{\frac{n}{2}}$ is nonlocal. In a recent work T. Jin, A. Maalaoui, L. Martinazzi, J. Xiong \cite{TALJ} have proven:
%the  existence of smooth solutions to \eqref{main-eq}-\eqref{V} in dimension $3$ for every $V\in (0,|S^3|)$. 

%[Theorem C JMMX]

\medskip
\noindent \textbf{Theorem C (\cite{TALJ})}
\emph{For every   $V\in (0,|S^3|)$ there exists at least one smooth solution to \eqref{main-eq}-\eqref{V}} with $n=3$.
 
 \medskip
 
Extending the results of \cite{CC}, \cite{W-Y}, \cite{H-M} and \cite{TALJ} to arbitrary odd dimension $n\ge 3$ we prove the following theorem about the existence of solutions to \eqref{main-eq}-\eqref{V}
with prescribed asymptotic behavior:

 \begin{thm}\label{thm-1}
Let $n\geq 3$ be an odd integer. For any given $V\in(0, |S^n|)$ and any given polynomial $P$ of degree at most $n-1$ such that
\begin{equation}\label{condP2}
P(x)\to\infty \quad \text{as } |x|\to\infty,
\end{equation}
there exists a $u\in C^{\infty}(\rn)\cap L_{\frac n2}(\rn)$ solution to \eqref{main-eq}-\eqref{V} having the asymptotic behavior given in \eqref{asmp} with $\alpha=\frac{2V}{|S^n|}$.
\end{thm}

Notice that, contrary to the result of Theorem C, in Theorem \ref{thm-1} we can now prescribe both the asymptotic  behaviour and the volume, similar to Theorem B, but in fact in more generality, since the condition \eqref{condP} has been replaced by the weaker condition \eqref{condP2}. Actually with minor modifications one can prove that the condition \eqref{condP2} also suffices in even dimension.
On the other hand we do not expect this assumption to be optimal, compare to Theorem D below.

 %on the polynomial $P$ is still not optimal. We expect to have a solution for any polynomial $P$ which satisfies
 %the conditions of the classification theorem by C.S. Lin and L. Martinazzi below:

We also remark that the condition $0<V<|S^n|$ is necessary for the existence of non-spherical solution to \eqref{main-eq}-\eqref{V} in dimension 3 and 4 as shown in \cite{TALJ} and \cite{Lin}
 respectively, but in higher dimension solutions could exist for large $V$, as shown for instance in dimension 6 by L. Martinazzi \cite{LM-vol}. %, who constructed rotationally symmetric solutions to \eqref{main-eq}-\eqref{V} with arbitrarily large $V$ using an ODE approach, which does not allow to prescribe the asymptotic behavior \eqref{asmp} (by choosing the polynomial $P$) and $V$ at the same time

Also the condition $n\ge 3$ in Theorem \ref{thm-1} is necessary, since for $n=1$ any solution of \eqref{main-eq}-\eqref{V} is spherical, i.e. as in \eqref{uspher}, see F. Da Lio, L. Martinazzi and T. Rivi\`ere \cite{DMR}.

\medskip

Now we move from the problem of existence to the problem of studying the most general asymptotic behavior of solutions to \eqref{main-eq}-\eqref{V}.

For $n$ even we have this result due to C.S. Lin for $n=4$ and L. Martinazzi when $n\ge 6$:
\medskip

\noindent \textbf{Theorem D (\cite{Lin}, \cite{LM})}
\emph{Any solution $u$ of \eqref{main-eq}-\eqref{V} with $n$ even has the asymptotic behavior 
\begin{align}\label{asymp}
u(x)=-P(x)-\alpha\log|x| +o(\log|x|)
\end{align}
where $\alpha=\frac{2V}{|S^n|}$, $\frac{o\left(\log|x|\right)}{\log|x|}\to 0$ as $|x|\to\infty$ and $P$ is a polynomial bounded from below and of degree at most $n-2$.}  

\medskip

Under certain regularity assumptions T. Jin, A. Maalaoui, L. Martinazzi, J. Xiong \cite{TALJ} extended the above result to dimension $3$, and among other things they proved:

\medskip
\noindent \textbf{Theorem E (\cite{TALJ})}
\emph{Let $u\in W^{2,1}_{loc}(\mathbb{R}^3)$ be such that  $\Delta u\in L_{\frac 12}(\mathbb{R}^3)$ and $u$ satisfies \eqref{V}. If $u$ solves \eqref{main-eq} in the  sense that
$$\int_{\mathbb{R}^3}(-\Delta)u(-\Delta)^{\frac 12}\varphi dx=2\int_{\mathbb{R}^3}e^{3u}\varphi dx,\quad\text{for every }\varphi\in\s(\mathbb{R}^3),$$  then $u$ has the asymptotic behavior 
given by \eqref{asymp} with $\alpha=\frac{2V}{|S^3|}$ and $P$ is a polynomial bounded from below and of degree $0$ or $2$.}  

\medskip

In our upcoming paper \cite{Hyd} extending Theorem E  we study the asymptotic behavior of solutions to \eqref{main-eq}-\eqref{V} in arbitrary odd dimension under much weaker regularity assumptions:

\medskip
\noindent \textbf{Theorem F(\cite{Hyd})}
\emph{Let $n\geq 3$ be any odd integer. Let $u\in L_{\frac n2}(\rn)$ be a solution to \eqref{main-eq}-\eqref{V} in the sense of Definition \ref{sol}. 
Then $u$ has the asymptotic behavior 
given by \eqref{asymp} with $\alpha=\frac{2V}{|S^n|}$ and $P$ a polynomial bounded from below and of degree at most $n-1$.}  

\medskip

Now we shall discuss the case when the $Q$-curvature is negative. We consider the equation 
\begin{align}\label{neg}
 (-\Delta)^{\frac{n}{2}}u=-(n-1)!e^{nu}\quad \text{in $\rn$}. 
\end{align}

 Geometrically a smooth solution of \eqref{neg} corresponds to a conformally flat metric $g_u=e^{2u}|dx|^2$ on $\rn$ which has constant $Q$-curvature $-(n-1)!$.
 
 In even dimension $n\geq 4$ L. Martinazzi \cite{LM-Neg} has shown that any solution $u$ to \eqref{neg}-\eqref{V} has the asymptotic behavior given by \eqref{asymp}
 with $\alpha=-\frac{2V}{|S^n|}$,  while for any $V>0$ and any given polynomial $P$ of degree at most $n-2$ satisfying \eqref{condP}, A. Hyder-L. Martinazzi \cite{H-M} have proven the existence of solutions
 to \eqref{neg}-\eqref{V} having the asymptotic behavior given by
 \eqref{asmp} with $\alpha=-\frac{2V}{|S^n|}$. As in the positive case, we shall extend this existence result to arbitrary odd dimension $n\geq 3$, again replacing condition \eqref{condP} with the weaker condition \eqref{condP2}.
 
\begin{thm}\label{thm-2}
Let $n\geq 3$ be an odd integer. For any given $V>0$ and any given polynomial $P$ of degree at most $n-1$ satisfying \eqref{condP2} there exists a 
$u\in C^{\infty}(\rn)\cap L_{\frac n2}(\rn)$ solution to \eqref{neg}-\eqref{V} having the asymptotic behavior given in \eqref{asmp} with $\alpha=-\frac{2V}{|S^n|}$.
 \end{thm}

Finally we remark that in dimension $1$ and $2$ \eqref{neg}-\eqref{V} has no solution (compare to \cite{LM-Neg} and \cite{DMR}).
 
%[Study of the asymptotic behaviour]

%State your theorem, comment that it extends the work \cite{TALJ}

%[Case of negative curvature: existence and asymptotic behaviour]

\medskip

\noindent\textbf{Acknowledgements} I would like to thank Prof. Luca Martinazzi for introducing me to this problem and his constant help throughout the work.

%%%%%%%%%%%%%%%%%%%%%%%%%%%%%%%%%%%%%%%%%%%%%%%%%%%%%%%%%%%%%%%%%%%%%%%%%%%%%%%%%%%%

 \section{Proof of Theorem \ref{thm-1} and Theorem \ref{thm-2}}
 
 The proof of Theorem \ref{thm-1} and \ref{thm-2} rests on the following  theorem:
\begin{thm}\label{thm-cv}
Let $w_0(x)=\log\frac{2}{1+|x|^2}$ and let $\pi:S^n\setminus \{N\}\to \rn$ be the stereographic projection and $N=(0,\dots,0,1)\in S^n$ be the North pole.
Take any number $\alpha\in (-\infty,0)\cup (0,2)$ and consider two functions $K,\varphi\in C^{\infty}(\rn)$ such that
 \begin{align}\label{cond-phi}
  \int_{\rn}\varphi dx=\gamma_n:=\frac{(n-1)!}{2}|S^n|, 
 \end{align}
$\alpha K> 0$ everywhere in $\rn$ and whenever $\alpha<0$ then $|K|>\delta e^{-\delta|x|^p}$ for some $\delta>0$, $0<p<n$. If both of $Ke^{-nw_0}$ and $\varphi e^{-nw_0}$ can be extend as $C^{2n+1}$ function
 on $S^n$ via the stereographic projection $\pi$  then the problem 
  \begin{equation}\label{eq_1}
   (-\Delta)^{\frac{n}{2}}w=Ke^{n(w+c_w)}-\alpha\varphi \quad  \text{in $\rn$}, \quad c_w:=-\frac{1}{n}\log\left(\frac{1}{\alpha\gamma_n}\int_{\rn}Ke^{nw}dx\right),%\int_{\rn}Ke^{nw}<\infty,
  \end{equation}
has at least one solution $w\in C^{\infty}(\rn)\cap L_{\frac n2}(\rn)$ (in the sense of Definition \ref{sol}) so that $\lim_{|x|\to\infty} w(x)\in\mathbb{R}$.   %and $\frac{1}{2}-c_0c_nn^2p>0$ (that is $\alpha p<2$) 
\end{thm}
%The assumption $(Ke^{-nw_0})\circ \pi$, $(\varphi e^{-nw_0})\circ\pi\in C^{2n+1}(S^n)$ is for 
Now the proof of Theorem \ref{thm-1} and Theorem \ref{thm-2} follows at once by taking $$u:=-P+\alpha u_0+w+c_w,$$ where 
$u_0\in C^{\infty}(\rn)$ is given by Lemma \ref{u_0} with $k=2n+3$, $w$ is the solution in Theorem \ref{thm-cv} with  $\varphi=(-\Delta)^{\frac{n}{2}}u_0$ which satisfies \eqref{cond-phi}
thanks to Lemma \ref{u_0bis}, and $K:=\sign(\alpha)(n-1)!e^{-nP+n\alpha u_0}$. Notice that $Ke^{-nw_0}$ can be extend smoothly on $S^n$ via the stereographic projection
$\pi$ where as $\varphi e^{-nw_0}$ can be extend as a $C^{2n+1}$ function. 

\begin{lem}\label{u_0} For every positive integer $k$ there exists a $u_0\in C^{\infty}(\rn)$ such that
\begin{align}
 u_0(x)=\log\frac{1}{|x|}\quad \text{for }\,|x|\geq 1\text{ and }\quad|D^{\alpha}(-\Delta)^{\frac{n}{2}}u_0(x)|\leq\frac{C}{|x|^{2n+k+|\alpha|}}\quad \text{for }x\ne 0,\label{condition}
\end{align}
for any multi-index $\alpha\in \mathbb{N}^n$.
\end{lem}
\begin{proof}
Inductively we define $$v_j(x)=\int_0^{x_1}v_{j-1}((t,\bar{x}))dt, \quad \text{for }x=(x_1,\bar{x})\in\mathbb{R}\times\mathbb{R}^{n-1},\,j=1,2,\dots,k,$$
where $$v_0(x)=\log\frac{1}{|x|}.$$
Let $\chi\in C^{\infty}(\rn)$ be such that 
 $$\chi(x)=\left\{\begin{array}{ll}
 0 & for \,|x|\leq \frac{1}{2}                   \\
 1 & for\, |x|\geq 1.
     \end{array}\right.$$ 
We claim that $u_0=\frac{\partial^k}{\partial x_1^k}(\chi v_k)$ satisfies \eqref{condition}. 
 It is easy to see that $u_0(x)=\log\frac{1}{|x|}$ for $|x|\geq 1$. By Lemma \ref{funda} $\frac{1}{\gamma_n}(-\Delta)^{\frac{n-1}{2}}\frac{\partial^k}{\partial x_1^k}v_k$ is a fundamental solution of 
 $(-\Delta)^{\frac{1}{2}}$ on $\rn$ and hence for $x\neq 0$, $(-\Delta)^{\frac{1}{2}}(-\Delta)^{\frac{n-1}{2}}\frac{\partial^k}{\partial x_1^k}v_k(x)=0$. For $|x|>2$ using integration by parts we compute  
 \begin{align}
  (-\Delta)^{\frac{n}{2}}u_0(x)&=(-\Delta)^{\frac{1}{2}}(-\Delta)^{\frac{n-1}{2}}\frac{\partial^k}{\partial x_1^k}(\chi v_k-v_k)(x)\notag\\
  &=C_n\int_{|y|<1}\frac{(-\Delta)^{\frac{n-1}{2}}\frac{\partial^k}{\partial y_1^k}(\chi v_k-v_k)(y)}{|x-y|^{n+1}}dy \notag\\
  &=C_n\int_{|y|<1}\left(\chi v_k-v_k\right)(y)\frac{\partial^k}{\partial y_1^k}(-\Delta)^{\frac{n-1}{2}}\left(\frac{1}{|x-y|^{n+1}}\right)dy,\notag
 \end{align}
 and
 $$D^{\alpha}(-\Delta)^{\frac{n}{2}}u_0(x)=C_n\int_{|y|<1}\left(\chi(y)v_k(y)-v_k(y)\right)D_x^{\alpha}\frac{\partial^k}{\partial y_1^k}(-\Delta)^{\frac{n-1}{2}}\left(\frac{1}{|x-y|^{n+1}}\right)dy.$$
Hence $$|D^{\alpha}(-\Delta)^{\frac{n}{2}}u_0(x)|\leq C\frac{\|v_k\|_{L^1(B_1)}}{|x|^{2n+k+|\alpha|}}.$$
\end{proof}

\begin{lem}\label{u_0bis}
Let $u_0\in C^{\infty}(\rn)$ be as given by Lemma \ref{u_0} for a given $k\in\mathbb{N}$. Then $$\int_{\rn}(-\Delta)^{\frac{n}{2}}u_0(x)dx=\gamma_n.$$
\end{lem}
\begin{proof}
 Let $\eta\in C^{\infty}(\rn)$ be such that $$\eta(x)=\left\{\begin{array}{ll}
                                                             1\quad \text{if } |x|\leq 1\\
                                                             0\quad\text{if } |x|\geq 2.
                                                             \end{array}
  \right.$$
  We set $\eta_k(x)=\eta(\frac xk)$. Then noticing that $(-\Delta)^{\frac{n}{2}}u_0\in L^1(\rn)$ one has
  \begin{align}
   \int_{\rn}(-\Delta)^{\frac{n}{2}}u_0(x)dx &=\lim_{k\to\infty}\int_{\rn}(-\Delta)^{\frac{n}{2}}u_0(x)\eta_k(x)dx \notag\\
   &= \lim_{k\to\infty}\int_{B_1}\left(u_0(x)-\log\frac{1}{|x|}\right)(-\Delta)^{\frac{n}{2}}\eta_k(x)dx  +\gamma_n     \notag\\
   &=\gamma_n, \notag
  \end{align}
where in the second equality we used the fact that $\frac{1}{\gamma_n}\log\frac{1}{|x|}$ is a fundamental solution of $(-\Delta)^{\frac{n}{2}}$ and
 the third equality follows from the locally uniform convergence of  $(-\Delta)^{\frac{n}{2}}\eta_k\to 0$.
 \end{proof}

 It remains to prove Theorem \ref{thm-cv}. In order to do that we recall the definition of $H^n(S^n)$ . 
 \begin{defn}
 Let $n\geq 3$ be an odd integer.
  Let $\{Y^m_l\in C^{\infty}(S^n): 1\leq m\leq N_l,\, l=0,1,2,\dots\}$ 
 be a orthonormal basis of $L^2(S^n)$ where $Y^m_l$ is an eigenfunction of the Laplace-Beltrami operator $-\Delta_{g_0}$ ($g_0$ denotes the round metric on $S^n$) corresponding to the eigenvalue $\lambda_l=l(l+n-2)$ and $N_l$ is the multiplicity of  
 $\lambda_l$. The space $H^n(S^n)$ is defined by %(\text{\bf{should I write $n\geq 3$ odd integer?}})
 $$H^n(S^n)=\left\{u\in L^2(S^n): \|u\|_{\dot H^n(S^n)}<\infty\right\},$$ where
 for any 
 $$u=\sum_{l=0}^{\infty}\sum_{m=1}^{N_l}u^m_lY^m_l$$
 we set
 $$\|u\|_{\dot H^n(S^n)}^2:=\sum_{l=0}^{\infty}\sum_{m=1}^{N_l}\left(\lambda_l+\left(\frac{n-1}{2}\right)^2\right)\prod_{k=0}^{\frac{n-3}{2}}
 (\lambda_l+k(n-k-1))^2(u^m_l)^2.$$
 \end{defn}
Notice that the norm $\|u\|_{\dot H^n(S^n)}^2$ is equivalent to the simpler norm $\|u\|^2:= \sum_{l=0}^\infty\sum_{m=1}^{N_l}\lambda_l^n (u_l^m)^2$, but  has the advantage of taking the form
$$\|u\|_{\dot H^n(S^n)}=\|P_{g_0}^n u\|_{L^2(S^n)},$$
where for $n$ odd the Paneitz operator $\Pg$ can be defined on $H^n(S^n)$ by 
 $$\Pg u=\sum_{l=0}^{\infty}\sum_{m=1}^{N_l}\left(\lambda_l+\left(\frac{n-1}{2}\right)^2\right)^{\frac 12}\prod_{k=0}^{\frac{n-3}{2}}
 (\lambda_l+k(n-k-1))u^m_lY^m_l.$$
 %$$
%\Pg=\left\{\begin{array}{ll}
%    \Pi_{k=0}^{\frac{n-2}{2}}\left(-\Delta_{g_0}+k(n-k-1)\right)\quad & \text{for $n$ even   }     \\
%    \left(-\Delta_{g_0}+\left(\frac{n-1}{2}\right)^2\right)^{\frac{1}{2}}\Pi_{k=0}^{\frac{n-3}{2}}\left(-\Delta_{g_0}+k(n-k-1)\right)\quad & \text{for $n$ odd.} 
%           \end{array}\right.
%$$
Since the operator $\Pg$ is positive we can define its square root, namely 
$$(\Pg)^{\frac 12}u:=\sum_{l=0}^{\infty}\sum_{m=1}^{N_l}\left(\lambda_l+\left(\frac{n-1}{2}\right)^2\right)^{\frac 14}\prod_{k=0}^{\frac{n-3}{2}}
 (\lambda_l+k(n-k-1))^{\frac 12}u^m_lY^m_l, \quad u\in H^\frac n2(S^n),$$
 where the space $H^\frac n2(S^n)$ is defined by $$H^\frac n2(S^n):=\left\{u\in L^2(S^n):\sum_{l=0}^{\infty}\sum_{m=1}^{N_l}\left(\lambda_l+\left(\frac{n-1}{2}\right)^2\right)^\frac 12\prod_{k=0}^{\frac{n-3}{2}}
 (\lambda_l+k(n-k-1))(u^m_l)^2<\infty\right\},$$ endowed with the norm 
 \begin{align}\|u\|^2_{H^\frac n2(S^n)}&:=\|u\|^2_{L^2(S^n)}+\|u\|^2_{\dot{H}^\frac n2(S^n)}\notag\\
 &:=\|u\|_{L^2(S^n)}^2 + \|(\Pg)^{\frac 12}u \|_{L^2(S^n)}^2 \notag%  \sum_{l=0}^{\infty}\sum_{m=1}^{N_l}\left(\lambda_l+\left(\frac{n-1}{2}\right)^2\right)^\frac 12\prod_{k=0}^{\frac{n-3}{2}} (\lambda_l+k(n-k-1))(u^m_l)^2. \notag
 \end{align}
 
 \begin{defn}\label{def-weak}
  Let $f\in H^{-\frac n2}(S^n)$ the dual of $H^{\frac n2}(S^n)$. A function $u\in H^{\frac n2}(S^n)$ is said to be a weak solution of $$\Pg u=f,$$ if 
  \begin{align}\label{weakd}
   \int_{S^n}(\Pg)^\frac12u(\Pg)^\frac12\varphi dV_0=\langle f,\varphi \rangle, \quad \text{for every }\varphi\in H^{\frac n2}(S^n).
  \end{align}

 \end{defn}

%where $\Delta_{g_0}$ is the Laplace Beltrami operator and $g_0$ is the standard round metric on $S^n$.
%\begin{lem}
 %Let $\pi:S^n\to\rn$ be the stereographic projection. Then for all $u\in H^n(S^n)$
 %$$(\Pg u)\circ\pi^{-1}=e^{-nw_0}(-\Delta)^{\frac{n}{2}}(u\circ\pi^{-1}),$$
 %where $w_0(x)=\log\frac{2}{1+|x|^2}.$
%\end{lem}

 The following estimate of Beckner is crucial in the proof of Theorem \ref{thm-cv}.
\begin{thm}[\cite{Bec}]\label{bec}
  For every $u\in H^{\frac{n}{2}}(S^n)$ one has 
\begin{align} 
 \log\left(\frac{1}{|S^n|}\int_{S^n}e^{u-\overline{u}}dV_0\right)\leq \frac{1}{2|S^n|n!}\int_{S^n}|(\Pg)^{\frac{1}{2}} u|^2dV_0,
\quad \overline{u}:=\frac{1}{|S^n|}\int_{S^n}udV_0.\notag
\end{align}
%$$\log\left(\frac{1}{|S|^n}\int_{S^n}e^{u-\overline{u}}dV_0\right)\leq c_n\|u\|^2_{\dot{H}^{\frac{n}{2}}}.$$
\end{thm}

\medskip

\noindent\emph{Proof of Theorem \ref{thm-cv}.}
 Let  $\tilde{K}=K\circ\pi$, $\varphi_1=\varphi e^{-nw_0}$ and $\tilde{\varphi_1}=\varphi_1\circ\pi$. Define the functional $J$ on 
$H^{\frac{n}{2}}(S^n)$ by
$$J(w):=\int_{S^n}\left(\frac{1}{2}|(\Pg)^{\frac{1}{2}}w|^2+\alpha\tilde{\varphi_1}w\right)dV_0-
                    \frac{\alpha\gamma_n}{n}\log\left(\int_{S^n}|\tilde{K}|e^{nw}e^{-nw_0\circ\pi}dV_0\right).$$ 
Using Theorem \ref{bec} we bound
\begin{align}\label{beck}
 \log\left(\int_{S^n}|\tilde{K}|e^{nw}e^{-nw_0\circ\pi}dV_0\right) &
 = \left(\log\left(\frac{1}{|S^n|}\int_{S^n}e^{nw-n\overline{w}}
 |\tilde{K}|e^{-nw_0\circ\pi}dV_0\right) + n\overline{w}+C\right) \notag \\
 & \leq \log\left(\frac{1}{|S^n|}\int_{S^n}e^{nw-n\overline{w}}dV_0\right)+\log\left(\|\tilde{K}e^{-nw_0\circ\pi}\|_{L^{\infty}}\right)
             +n\overline{w}+C \notag\\
  &\leq \frac{n^2}{2|S^n|n!}\int_{S^n}|(\Pg)^{\frac{1}{2}}w|^2dV_0 +n\overline{w}+C.         
\end{align}
Since for any $c\in \mathbb{R}$  $J(w+c)=J(c)$ we can assume $\overline{w}=0$. Then from (\ref{beck}) we have
\begin{align}
 J(w) \geq \min\bigg\{\frac {1}{2}, \bigg(\underbrace{\frac{1}{2}-\frac{\alpha\gamma_n}{n}\frac{n^2}{2|S^n|n!}}_{=(2-\alpha)/4}\bigg)\bigg\}\|w\|_{\dot{H}^{\frac{n}{2}}}^2-\varepsilon\|w\|_{\dot{H}^{\frac{n}{2}}}^2-
              \frac{1}{\varepsilon}\|\tilde{\varphi}_1\|_{L^2}^2-C, \notag
\end{align}
where $0<\varepsilon<\frac 12$  is sufficiently small such that $\frac{2-\alpha}{4}-\varepsilon>0$ and for $\alpha<0$ using 
$|K|>\delta e^{-\delta|x|^p}$ one has
$$\log\left(\int_{S^n}|\tilde{K}|e^{nw}e^{-nw_0\circ\pi}dV_0\right) \geq \frac{1}{|S^n|}\int_{S^n}\log\left(|\tilde{K}|e^{-nw_0\circ\pi}\right)dV_0
+n\overline{w}+\log|S^n|\geq-C.$$ %where $S:=\{x\in S^n:\tilde{K}(x)\neq0\}.$
Thus a minimizing sequence $\{w_k\}$
of $J$ with $\overline{w}_k=0$ is bounded in ${\dot{H}^{\frac{n}{2}}({S^n})}$.  With the help of Poincar\'e inequality
$$\|w-\overline{w}\|_{L^2(S^n)}\leq \|(\Pg)^{\frac 12}w\|_{L^2(S^n)},\quad\text{for every }w\in H^{\frac n2}(S^n),$$ 
which easily follows from the definition of $\|(\Pg)^{\frac 12}w\|_{L^2(S^n)}$,  we conclude that the sequence $\{w_k\}$ is bounded in ${H^{\frac{n}{2}}({S^n})}$.
Then up to a subsequence $w_k$ converges weakly to $u$ for some $u\in {H^{\frac{n}{2}}({S^n})}$. 
From the compactness of the map $v\mapsto e^v$ from ${H^{\frac{n}{2}}({S^n})}$ to $L^P(S^n)$ for any $p\in [1,\infty)$ (for a simple proof see Proposition 7 in \cite{TALJ} which holds
in higher dimension as well) we have (up to a subsequence)
$$\lim_{k\to\infty}\log\left(\int_{S^n}|\tilde{K}|e^{nw_k}e^{-nw_0\circ\pi}dV_0\right)=\log\left(\int_{S^n}|\tilde{K}|e^{nu}e^{-nw_0\circ\pi}dV_0\right).$$
Moreover from the weak convergence of $w_k$ to $u$ we have
$$\lim_{k\to\infty}\int_{S^n}\tilde{\varphi}_1w_kdV_0=\int_{S^n}\tilde{\varphi}_1udV_0 \quad\text{and} \quad \|u\|_{H^{\frac n2}(S^n)}\leq\liminf_{k\to\infty}\|w_k\|_{H^{\frac n2}(S^n)},$$
and from the compact embedding $H^{\frac n2}(S^n)\hookrightarrow L^2(S^n)$ we get $$\lim_{k\to\infty}\|w_k\|_{L^2(S^n)}=\|u\|_{L^2(S^n)}.$$
Thus $\|(\Pg)^{\frac 12}u\|_{L^2(S^n)}\leq \liminf_{k\to\infty}\|(\Pg)^{\frac 12}w_k\|_{L^2(S^n)}$ which implies that $u$ is a minimizer of $J$ and hence
  $u$ is a weak solution of (in the sense of Definition \ref{def-weak})
$$\Pg u+\alpha\tilde{\varphi_1}=\frac{\alpha\gamma_n}{\int_{S^n}\tilde{K}e^{nu}e^{-nw_0}dV_0}\tilde{K}e^{-nw_0\circ\pi}e^{nu}
=:C_0\tilde{K}e^{-nw_0\circ\pi}e^{nu}.$$ Since $\tilde{\varphi_1}\in C^{2n+1}(S^n)$ and $\tilde{K}e^{-nw_0\circ\pi}\in C^{\infty}(S^n)$ we have 
$$\Pg u=C_0\tilde{K}e^{-nw_0\circ\pi}e^{nu}-\alpha\tilde{\varphi_1}\in L^2(S^n),$$ and by Lemma \ref{Hn} below $u\in H^n(S^n)$  %Now
%$u\in H^n(S^n)\hookrightarrow C^{\frac{n-1}{2},\frac 12}(S^n)$ implies $\Pg u\in H^{\frac n2}(S^n)$ and a repeated use of Lemma \ref{regularity} gives 
and a repeated use of Lemma \ref{regularity} gives $u\in C^{2n+1}(S^n)$.

We set  $w:=u\circ\pi^{-1}$ and $w_k:=u_k\circ\pi^{-1}$ where $u_k\in C^{\infty}(S^n)$ be such that $u_k \xrightarrow {C^{2n+1}(S^n)}u$. 
It is easy to see that $(-\Delta)^{\frac{n}{2}}w _k\xrightarrow{C^0(\rn)} (-\Delta)^{\frac{n}{2}}w$ and $\Pg u_k\xrightarrow{C^0(S^n)}\Pg u$ which easily 
follows from $$\Pg u_k\xrightarrow{H^{n+1}(S^n)}\Pg u \text{ and } H^{n+1}(S^n)\hookrightarrow C^0(S^n).$$

 Now using the following identity of T. Branson (see \cite{TB})
 $$(-\Delta)^{\frac{n}{2}}(v\circ\pi^{-1})=e^{nw_0}(\Pg v)\circ\pi^{-1}\quad \text{for every }v\in C^{\infty}(S^n),$$ we get
\begin{align}
 (-\Delta)^{\frac{n}{2}}w =(-\Delta)^{\frac{n}{2}}(u\circ\pi^{-1}) 
   &=e^{nw_0}\left(C_0\tilde{K}e^{-nw_0\circ\pi}e^{nu}-\alpha\tilde{\varphi_1}\right)\circ\pi^{-1} \notag\\
  & %= C_0Ke^{nw}-\alpha e^{nw_0}\varphi_1
  =C_0Ke^{nw}-\alpha\varphi=Ke^{n(w+c_w)}-\alpha\varphi. \notag 
\end{align}
Since $(-\Delta)^{\frac{n}{2}}w\in L_{\frac 12}(\rn)\cap C^{2n+1}(\rn)$ we have 
$$(-\Delta)^{\frac{n+1}{2}}w=(-\Delta)^\frac12(-\Delta)^{\frac{n}{2}}w\in C^{2n}(\rn),$$ and by bootstrap argument we conclude that $w\in C^{\infty}(\rn)$.

%Note that
%\begin{align}
% \int_{\rn}Ke^{nw}&=\int_{\rn}Ke^{nu\circ\pi^{-1}} 
 % =\int_{S^n}\tilde{K}e^{nu}e^{-nw_0\circ\pi}dV_0<\infty,\notag
%\end{align}
%and hence $w$ is a solution of (\ref{eq_1}) as 
%$$C_0=\frac{\alpha\gamma_n}{\int_{S^n}\tilde{K}e^{nu}e^{-nw_0}dV_0}=\frac{\alpha\gamma_n}{\int_{\rn}Ke^{nw}}=e^{nC_w}.$$
\hfill $\square$

\begin{lem}\label{Hn}
 Let $f\in L^2(S^n)$. Let $u\in H^{\frac n2}(S^n)$ be a weak solution (in the sense of Definition \ref{def-weak}) of
 $$\Pg u=f \quad\text{on } S^n.$$ Then $u\in H^n(S^n)$.
\end{lem}
\begin{proof}
Let $$u=\sum_{l=0}^{\infty}\sum_{m=1}^{N_l}u^m_lY^m_l\quad  \text{ and } f=\sum_{l=0}^{\infty}\sum_{m=1}^{N_l}f^m_lY^m_l.$$ 
Taking the test function $\varphi=Y^m_l$ in \eqref{weakd} we get 
%From the definition of weak solution we have 
$$f^m_l=\int_{S^n}(\Pg)^\frac12u(\Pg)^\frac12\varphi dV_0=\left(\lambda_l+\left(\frac{n-1}{2}\right)^2\right)^{\frac 12}\prod_{k=0}^{\frac{n-3}{2}}
 (\lambda_l+k(n-k-1))u^m_l.$$
Hence 
\begin{align}
\|\Pg u\|_{L^2(S^n)}=\sum_{l=0}^{\infty}\sum_{m=1}^{N_l}\left(\lambda_l+\left(\frac{n-1}{2}\right)^2\right)\prod_{k=0}^{\frac{n-3}{2}}
 (\lambda_l+k(n-k-1))^2(u^m_l)^2=\sum_{l=0}^{\infty}\sum_{m=1}^{N_l}(f^m_l)^2<\infty,\notag
  \end{align}
and we conclude the proof.
\end{proof}

\begin{lem}\label{regularity}
 Let $u\in H^s(S^n)$ and $f\in H^{s-n+t}(S^n)$ for some $s\geq n$ and $t\geq 0$. If $u$ solves 
 \begin{align}
 \Pg u=f \quad\text{on } S^n,\notag
  \end{align}
 then $u\in H^{s+t}(S^n)$. 
\end{lem}
\begin{proof}
 Let $$u=\sum_{l=0}^{\infty}\sum_{m=1}^{N_l}u^m_lY^m_l,$$ and $$(-\Delta_{g_0})^{\frac{s-n}{2}}f=:h=\sum_{i=0}^{\infty}\sum_{j=1}^{N_i}h^j_iY^j_i,$$%\in H^l(S^n),$$
 where for any $r>0$
 $$(-\Delta_{g_0})^rv=\sum_{l=0}^{\infty}\sum_{m=1}^{N_l}v^m_l\lambda_l^rY^m_l \quad \text{for } v=\sum_{l=0}^{\infty}\sum_{m=1}^{N_l}v^m_lY^m_l\in H^{2r}(S^n).$$
 Then 
 \begin{align}
  (-\Delta_{g_0})^{\frac{s-n}{2}}\Pg u=h \quad\text{on } S^n.\label{1}
 \end{align}
 Multiplying both sides of \eqref{1} by $Y^i_j$ and integrating on $S^n$ one has
 \begin{align}
  \left(\lambda_j+\left(\frac{n-1}{2}\right)^2\right)^{\frac 12}\prod_{k=0}^{\frac{n-3}{2}}
 (\lambda_j+k(n-k-1))\lambda_j^{\frac{s-n}{2}}u^i_j=h^i_j. \notag
\end{align}
Since $h\in H^t(S^n)$ we have 
$$\sum_{l=0}^{\infty}\sum_{m=1}^{N_l}\left(\lambda_l+\left(\frac{n-1}{2}\right)^2\right)\prod_{k=0}^{\frac{n-3}{2}}
 (\lambda_l+k(n-k-1))^2\lambda_l^{s-n}\lambda_l^t(u^m_l)^2<\infty,$$ and hence $u\in H^{s+t}(S^n)$.
\end{proof}

%%%%%%%%%%%%%%%%%%%%%%%%%%%%%%%%%%%%%%%%%%%%%%%%%%%%%%%%%%%%%%%%%%%%%%%%%%%
\appendix
 \section{Appendix}

  \begin{lem}[Fundamental solution]\label{funda} For $n\ge 3$ odd integer the function 
  $$\Phi(x):=\frac{(\frac{n-3}{2})!}{2\pi^{\frac{n+1}{2}}}\frac{1}{|x|^{n-1}}=\frac{1}{\gamma_n}(-\Delta)^{\frac{n-1}{2}}\log\frac{1}{|x|}$$ is a fundamental solution of $(-\Delta)^{\frac{1}{2}}$ in $\rn$ in 
 the sense that for all 
 $f\in L^1(\rn)$ we have  $\Phi\ast f\in L_{\frac{1}{2}}(\rn)$ and for all $\varphi\in\s(\rn)$
 \begin{align}\label{funda-2}
  \int_{\rn}(-\Delta)^{\frac{1}{2}}(\Phi\ast f)\varphi dx:=\int_{\rn}(\Phi\ast f)(-\Delta)^{\frac{1}{2}}\varphi dx=\int_{\rn}f\varphi dx.
 \end{align}
\end{lem}
\begin{proof}
To show $\Phi*f\in L_{\frac 12}(\rn)$ we bound 
\begin{align}
 \int_{\rn}\frac{|\Phi*f(x)|}{1+|x|^{n+1}}dx &\leq C\int_{\rn}\frac{1}{1+|x|^{n+1}}\left(\int_{\rn}\frac{1}{|x-y|^{n-1}}|f(y)|dy\right)dx \notag\\
 &=C\int_{\rn}|f(y)|\left(\int_{\rn}\frac{1}{1+|x|^{n+1}}\frac{1}{|x-y|^{n-1}}dx\right)dy \notag\\
 &\leq C\int_{\rn}|f(y)|\left(\int_{B_1}\frac{dx}{|x|^{n-1}}+\int_{\rn}\frac{dx}{1+|x|^{n+1}}\right)dy \notag\\
 &\leq C\|f\|_{L^1(\rn)}. \label{funda_1}
\end{align}
%Hence $\Phi\ast f\in L_{\frac{1}{2}}(\rn)$.
  
 If $f\in\ C_c^{\infty}(\rn)$ then \eqref{funda-2} is true by Theorem 5.9 in \cite{Lieb}. For the general case $f\in L^1(\rn)$ choose $f_k\in C_c^{\infty}(\rn)$ such that
 $f_k\to f$ in $L^1(\rn)$. Then using \eqref{funda_1} with $f\equiv f_k-f$ one has
 $$\int_{\rn}|\Phi\ast (f_k-f)||(-\Delta)^{\frac{1}{2}}\varphi| dx\leq C\int_{\rn}\frac{|\Phi\ast (f_k-f)(x)|}{1+|x|^{n+1}}dx
 \leq C\|f_k-f\|_{L^1(\rn)}\to 0,$$ that is 
 $$\int_{\rn}(\Phi\ast f_k)(-\Delta)^{\frac{1}{2}}\varphi dx\to\int_{\rn}(\Phi\ast f)(-\Delta)^{\frac{1}{2}}\varphi dx.$$ 
Now the proof follows from $$\int_{\rn}(\Phi\ast f_k)(-\Delta)^{\frac{1}{2}}\varphi dx=\int_{\rn}f_k\varphi dx\to\int_{\rn}f\varphi dx.$$ 
 \end{proof}
 
 Proof of the following Proposition can be found in \cite{Hyd}.
 \begin{prop}\label{ests} For any $s>0$ and $\varphi\in \s(\rn)$ we have
$$|(-\Delta)^s \varphi(x)|\le \frac{C}{|x|^{n+2s}}.$$
%where $(-\Delta)^s\varphi :=(-\Delta)^\sigma\circ (-\Delta)^k\varphi$, where $\sigma\in [0,1)$, $k\in \mathbb{N}$ and $s=k+\sigma$.
\end{prop}


\begin{thebibliography}{10}
\small

\bibitem{Bec}
\textsc{W. Beckner:}
\emph{Sharp Sobolev inequalities on the sphere and the Moser-Trudinger inequality},
Ann. of Math. (2) \textbf{138} (1993), 213-242.

\bibitem{TB}
\textsc{T. Branson:} 
\emph{Sharp inequality, the functional determinant and the complementary series}, 
Trans. Amer. Math. Soc. \textbf{347} (1995), 3671-3742.

%\bibitem{BM} 
%\textsc{H. Br\'ezis, F. Merle:}
%\emph{Uniform estimates and blow-up behaviour for solutions of $-\Delta u=V(x)e^u$ in two dimensions},
%Comm. Partial Differ. Equ. \textbf{16}, 1223–1253 (1991)



\bibitem{CC} \textsc{S-Y. A. Chang, W. Chen:} \emph{A note on a class of higher order conformally covariant equations},
Discrete Contin. Dynam. Systems \textbf{63} (2001), 275-281.


%\bibitem{C-Y}
%\textsc{S-Y. A. Chang, P. C. Yang:}
%\emph{On uniqueness of solutions of $n$-th order differential equations in conformal geometry},
%Math. Res. Lett. \textbf{4} (1997), 91-102.

 \bibitem{Chen-Li}
 \textsc{W. Chen, C. Li:}
 \emph{Classification of solutions of some nonlinear elliptic equations}, Duke Math. J. \textbf{63} (3) (1991), 615-622. 

\bibitem{DMR}
\textsc{F. Da Lio, L. Martinazzi, T. Rivi\`ere:}
\emph{Blow-up Analysis of a nonlocal Liouville-type equation}, preprint (2014).

%\bibitem{Valdinoci} \textsc{Di Nezza, Eleonora; Palatucci, Giampiero; Valdinoci, Enrico:} \emph{Hitchhiker's guide to the fractional Sobolev spaces},
 %Bull. Sci. Math. \textbf{136} (2012), no. 5, 521–573.  


\bibitem{Hyd} \textsc{A. Hyder:} \emph{Structure of conformal metrics on $\rn$ with constant Q-curvature}, preprint (2015).


\bibitem{H-M}
\textsc{A. Hyder, L. Martinazzi:}
\emph{Conformal metrics on $\mathbb{R}^{2m}$ with constant $Q$-curvature, prescribed volume and asymptotic behavior},
Discrete Contin. Dynam. Systems A \textbf{35} (2015), no.1, 283-299.

\bibitem{TALJ}
\textsc{T. Jin, A. Maalaoui, L. Martinazzi, J. Xiong:}
\emph{Existence and asymptotics for solutions of a non-local $Q$-curvature equation in dimension three},
Calc. Var. Partial Differential Equations (2014), DOI:10.1007/s00526-014-0718-9.

%\bibitem{Landkof} \textsc{N.S. Landkof:} \emph{Foundations of modern potential theory},
%Springer-Verlag, New York-Heidelberg, 1972. 

 \bibitem{Lieb}
 \textsc{E.H. Lieb, M. Loss:}
 \emph{Analysis},
 Second edition. Graduate Studies in Mathematics, 14. American Mathematical Society, Providence, RI, 2001. ISBN:0-8218-2783-9.

 \bibitem{Lin}
 \textsc{C.S. Lin:}
 \emph{A classification of solutions of a conformally invariant fourth order equation in $\mathbb{R}^n$}, 
 Comment. Math. Helv. \textbf{73} (1998), no. 2, 206-231.
 
 \bibitem{LM-Neg}
 \textsc{L. Martinazzi:}
 \emph{Conformal metrics on $\mathbb{R}^{2m}$ with constant $Q$-curvature},
 Rend. Lincei. Mat. Appl. \textbf{19} (2008), 279-292.
 
\bibitem{LM}
\textsc{L. Martinazzi:}
\emph{Classification of solutions to the higher order Liouville's equation on $\mathbb{R}^{2m}$ }, 
Math. Z. \textbf{263} (2009), no. 2, 307-329.

\bibitem{LM-vol}
\textsc{L. Martinazzi:}
\emph{Conformal metrics on $\mathbb{R}^{2m}$ with constant $Q$-curvature and large volume},
Ann. Inst. Henri Poincar\'e (C) \textbf{30} (2013), 969-982.


 
%\bibitem{LS} \textsc{L. Silvestre:} \emph{Regularity of the obstacle problem for a fractional power of the Laplace operator},
%Comm. Pure Appl. Math. \textbf{60} (2007), no. 1, 67–112.

\bibitem{W-Y}
\textsc{J. Wei, D. Ye:}
\emph{Nonradial solutions for a conformally invariant fourth order equation in  $\mathbb{R}^{4}$},
Calc. Var. Partial Differential Equations  \textbf{32} (2008), no. 3, 373-386.

%\bibitem{Xu} \textsc{X. Xu:}, 
%\emph{Uniqueness and non-existenceUniqueness and non-existence theorems for conformally invariant equations},
 %J.Funct. Anal. \textbf{222} (2005), no. 1, 1–28.






\end{thebibliography}
\end{document}